\documentclass[11pt]{amsart}
\usepackage{fancyhdr}
\usepackage{txfonts}

\usepackage{amsmath,amsthm,amssymb,hyperref, mdframed}
\usepackage{tikz}
\usepackage[all]{xy}
\usetikzlibrary{decorations.markings}
\usetikzlibrary{backgrounds,shapes}
\usepackage{subfig}

\usepackage{color}
\usepackage{marginnote}

\definecolor{aliceblue}{rgb}{0.9, 0.95, 1.0}

\newcommand{\pslc}{{\mathrm{PSL}_2 (\mathbb{C})}}
\newcommand{\pslnc}{{\mathrm{PSL}_n (\mathbb{C})}}
\newcommand{\pslnr}{{\mathrm{PSL}_n (\mathbb{R})}}
\newcommand{\pslr}{{\mathrm{PSL}_2 (\mathbb{R})}}

\newcommand\HH{{\mathbb H}}

\newcommand{\cp}{\mathbb{C}\mathrm{P}^1}

\newtheorem{theorem}{Theorem}[section]
\newtheorem{prop}[theorem]{Proposition}

\newtheorem{cor}[theorem]{Corollary}
\newtheorem{conj}[theorem]{Conjecture}

\newtheorem{thm}{Theorem}[section]

\newtheorem{lem}[thm]{Lemma}
\theoremstyle{definition}
\newtheorem{defn}[thm]{Definition}

\title[Dominating surface-group representations into $\pslc$]{Dominating surface-group representations into $\pslc$  in the relative representation variety}

\author{Subhojoy Gupta}

\address{Department of Mathematics, Indian Institute of Science,
Bangalore, India}
\email{subhojoy@iisc.ac.in}

\author{Weixu Su}

\address{Fudan University, Shanghai, China}
\email{suwx@fudan.edu.cn}

\begin{document}

\begin{abstract}
Let $\rho$ be a representation of the fundamental group of a punctured surface into $\pslc$ that is not Fuchsian. We prove that there exists a Fuchsian representation that strictly dominates $\rho$ in the simple length spectrum, and preserves the boundary lengths. This extends a result of Gueritaud-Kassel-Wolff to the case of $\pslc$-representations. Our proof involves straightening the pleated plane in $\mathbb{H}^3$ determined by the Fock-Goncharov coordinates of a framed representation, and applying strip-deformations.
\end{abstract}

\maketitle

\section{Introduction}

Let $S_{g,k}$ be an oriented surface of genus $g\geq 0$ and $k\geq 1$ labelled punctures $p_1,p_2,\ldots p_k$, with negative Euler characteristic; let $\Pi$ denote its fundamental group.

Given a representation $\rho: \Pi \to \pslc$, we define the $\rho$-\textit{length} of a closed curve $\gamma \in  \Pi$ to  be the translation length  of $\rho(\gamma)$, that is, $l_\rho(\gamma) = \inf\limits_{x\in \mathbb{H}^3} d_{\mathbb{H}^3}(\rho(\gamma)\cdot x,x).$ This is determined by the trace of $\rho(\gamma)$, since $\text{tr}^2(\rho(\gamma)) = 4\cosh^2(l_\rho(\gamma)/2)$.  Note that the translation length is positive if $\rho(\gamma)$ is loxodromic and zero if $\rho(\gamma)$ is parabolic or elliptic. Moreover, if $\rho$ is  \textit{Fuchsian}, that is, can be conjugated to a discrete and faithful representation into $\pslr$, then $l_\rho(\gamma)$ coincides with the length of the geodesic representative of $\gamma$ on the hyperbolic surface $\HH^2/\Gamma$, where $\Gamma$ is the image of the conjugated representation.  

For a fixed $k$-tuple $\mathcal{L} = (l_1,l_2,\ldots l_k) \in \mathbb{R}_{\geq 0}^k$, the \textit{relative representation variety} for the surface-group $\Pi$ is the space of representations
\begin{equation*}
\text{Hom}(\Pi, \mathcal{L}) = \{ \rho:\Pi \to \pslc \ \vert\ l_\rho(\gamma_i) = l_i \text{ where } \gamma_i \text{ is the loop around } p_i \}
\end{equation*}
where note that each \textit{boundary length} is fixed, and prescribed by $\mathcal{L}$.

A Fuchsian representation $j \in \text{Hom}(\Pi, \mathcal{L})$ is said to \textit{dominate} a representation $\rho\in  \text{Hom}(\Pi, \mathcal{L})$ if
\begin{equation}\label{supl}
\sup\limits_{\gamma} \frac{l_\rho(\gamma)}{l_j(\gamma)} \leq 1
\end{equation}
where $\gamma$ varies over all non-peripheral essential simple closed curves on $S_{g,k}$.  Moreover, $j$ is said to \textit{strictly dominate} $\rho$ if the inequality in \eqref{supl} is strict. Note that these can also described as domination and strict domination \textit{in the simple length spectrum} --  \S2 for a discussion about alternative definitions. It follows from the work in \cite{Thurston-Minimal} that a Fuchsian representation cannot have a strictly dominating Fuchsian representation in the same relative representation variety.

In this note we prove:

\begin{thm}\label{thm1} Let $\mathcal{L} = (l_1,l_2,\ldots l_k) \in \mathbb{R}_{\geq 0}^k$. For any  $\rho \in  \text{Hom}(\Pi, \mathcal{L})$ that is not Fuchsian,  there exists a Fuchsian representation $j \in  \text{Hom}(\Pi, \mathcal{L})$ that strictly dominates $\rho$.
\end{thm}

Theorem \ref{thm1} is inspired by the work of  Gueritaud-Kassel-Wolff in \cite{GKW} who used ``folded" hyperbolic surfaces to prove a strict domination result in the complete length spectrum, for \textit{closed} surface-group representations into $\pslnr$. Around the same time, Deroin-Tholozan (\cite{DerThol}) proved a more general domination result, for representations of a closed surface-group into the isometry group of smooth Riemannian $CAT(-1)$ spaces, using the theory of harmonic maps.  Their result was extended to isometry groups of $CAT(-1)$  metric spaces in \cite{DMSV}, and more recently,  to the case of punctured surfaces by Sagman in  \cite{Sagman}.  In particular, Theorem \ref{thm1} is implicit in \cite{Sagman}; indeed, we rely on his work in the the case when $\rho$ is a degenerate and co-axial representation (see Definition \ref{degen}). In the non-degenerate case, however,  our proof of Theorem \ref{thm1} avoids harmonic maps, and relies instead on the pleated-surface interpretation of the Fock-Goncharov coordinates of a framed  representation into $\pslc$, as exploited in \cite{GM2} and \cite{GupGKM}.  \\

Our result also improves on Theorem 1.2 of  \cite{Whang}. Indeed, it is known that there is a Bers' constant  $B(g,k,L)>0$ for  any hyperbolic surface with geodesic boundary having lengths above  bounded by $L$  (see \cite{Matelski}). It follows directly from Theorem \ref{thm1} that

\begin{cor}\label{cor-thm1}
Let $\rho \in \text{Hom}(\Pi,  \mathcal{L})$  be a representation such that the length of each peripheral curve is bounded
above by $L$. Then there exists a pants decomposition of $S_{g,k}$ such that the $\rho$-lengths  of the pants curves are at most $B(g,k,L)$.
\end{cor}

We expect that Theorem \ref{thm1} and Corollary \ref{cor-thm1} would be helpful in the study of the relative representation variety. Moreover, we hope that our techniques can be extended to proving analogous results for representations of punctured surface-groups into other complex Lie groups $G$, like $\text{PSL}_n(\mathbb{C})$ for $n>2$,  since we do have Fock-Goncharov coordinates for such representations (see \cite{FG}). For any such representation, one can define domination exactly as in \eqref{supl}, where the $\rho$-length of $\gamma$ is defined as the translation-length of $\rho(\gamma)$ in the symmetric space $\text{PSL}_n(\mathbb{C})/\text{PSU}_n$ equipped with a suitable invariant metric.  From the work in \cite{FG}, when the Fock-Goncharov coordinates are real and positive, they determine a Hitchin representation with image in $\text{PSL}_n(\mathbb{R})$.  We conjecture:

\begin{conj} A generic representation $\rho:\Pi \to \pslnc$ has a strictly dominating Hitchin representation $j:\Pi \to \pslnr$, in the sense of \eqref{supl}, in the  same relative representation variety.
\end{conj}

See \cite{DaiLi} for a closely related ``Metric Domination Conjecture" in the case of a closed surface. \\

\smallskip

 \textbf{Acknowledgements.}  This paper was conceived during SG's visit to Fudan University in June 2019; he is also grateful for their hospitality and support. SG also thanks the SERB, DST (Grant no. MT/2017/000706), and the UGC Center for Advanced Studies grant for their support.  W. Su is partially supported by NSFC No: 11671092 and No: 11631010. We are grateful to Nathaniel Sagman for his comments on an earlier version of this paper, and for pointing out how to handle the degenerate and co-axial case in \S3.2.

\section{Preliminaries}

\subsection{Framed and non-degenerate representations}
Fix a choice of a finite-area hyperbolic metric  on $S={S}_{g,k}$, such that the punctures are cusps.  Passing to the universal cover $\widetilde{S} \cong \HH^2$, let the Farey set ${F}_\infty$ be the points in the ideal boundary that are the lifts of the punctures. Note that ${F}_\infty$ is equipped with an action of the surface-group $\Pi = \pi_1(S)$.

In this paper, a  \textit{framed representation} $\hat{\rho}$ is a pair $(\rho, \beta)$ where $\rho  \in \text{Hom}(\Pi, \pslc)$ and $\beta:{F}_\infty \to \cp$ is a $\rho$-equivariant map. (Throughout this article, subgroups of $\pslc$ are assumed to act on $\cp$ by M\"{o}bius transformations, and on hyperbolic $3$-space $\mathbb{H}^3$ by isometries.) 
  The $\rho$-equivariance implies that if one fixes  fundamental domain $F$ of the action of $\Pi$ on the universal cover $\widetilde{S}$, one obtains for  each puncture a choice of a fixed point on $\cp$  of the corresponding boundary monodromy.

 For the notion of a \textit{non-degenerate} framed representation, see Definition 2.6 of \cite{GupGKM} or  Definition 4.3 of \cite{AllBrid}.  The relevant fact for us is that given a non-degenerate framed representation $\hat\rho$, the representation $\rho$ obtained by forgetting the framing is a  \textit{non-degenerate} representation, defined as follows (see Definition 2.4 of \cite{GupGKM}):

 \begin{defn}\label{degen}
A representation $\rho:\Pi\to \pslc$ is said to be \textit{degenerate} if either
\begin{itemize}
\item[(a)] the image of $\rho$ has a global fixed point on $\cp$, and $\rho(\gamma_i)$  is parabolic or identity for each peripheral loop $\gamma_i$, or
\item[(b)] the image of $\rho$ preserves a two-point set on $\cp$, which is fixed by each $\rho(\gamma_i)$  (where $1\leq i\leq k$). In this case $\rho$ is said to be \textit{co-axial} since its image would preserve a geodesic line in $\mathbb{H}^3$.
\end{itemize}
A representation is then said to be \textit{non-degenerate} if it is not degenerate.
\end{defn}

Note that it follows from this definition that a non-elementary representation is automatically non-degenerate; however there are elementary representations that are non-degenerate -- see \S2.4 of \cite{GupGKM}. Conversely, given a non-degenerate representation $\rho$, one can construct a non-degenerate framing $\beta$ by assigning to each puncture one of the fixed points of the holonmy/monodromy around it  -- see Proposition 3.1 of \cite{GupGKM}.

 \subsection{Fock-Goncharov coordinates}

An ideal triangulation $T$ on $S$ is a collection of geodesic arcs between cusps (the \textit{edges}), each homotopically non-trivial and non-peripheral, such that the complementary regions are ideal triangles. Given a  choice of an ideal triangulation $T$, and a \textit{generic} framed representation $\hat{\rho}$, one can define non-zero complex numbers to each edge $e \in T$ as follows: consider a lift $\tilde{e}$, and consider the two ideal triangles $\Delta_+$ and $\Delta_-$ in the lifted ideal triangulation that share the side $\tilde{e}$. The ideal vertices of $\Delta_\pm$ determine four points in $F_\infty$, and the image of these under $\beta$ will determine four distinct points in $\cp$ (here we use the genericity assumption). The \textit{Fock-Goncharov coordinate} $c(e) \in \mathbb{C}^\ast$ is then the cross-ratio of these four points.

Let $\widehat{\chi}(S)$ be the \textit{moduli space of framed representations}, namely, the set of framed representations up to the equivalence relation $(\rho, \beta)\sim (A\rho A^{-1}, A \cdot \beta)$ for any $A\in \pslc$. Then  Fock-Goncharov showed that the assignment  $[\hat{\rho}] \mapsto \{ c(e)\} _{e\in T}$ defines a birational isomorphism $\Phi_T: \widehat{\chi}(S)\to (\mathbb{C}^\ast)^{\lvert T\rvert}$ (see Theorem 1 of \cite{FG}).

We shall use the following result of Allegretti-Bridgeland:

\begin{thm}[Theorem 9.1 of \cite{AllBrid}]\label{thm-ab} For a  non-degenerate framed representation $\hat{\rho}$, there is an  ideal triangulation $T$ such that the Fock-Goncharov coordinates for $\hat{\rho}$ are well-defined and non-zero.
\end{thm}

 \subsection{Pleated planes in $\mathbb{H}^3$}
  A \textit{pleated plane} in $\mathbb{H}^3$ is a map $$\Psi:\widetilde{S} \to \mathbb{H}^3$$
that is totally geodesic on each lift of an ideal triangle on $S$ determined by an ideal triangulation $T$.  The image of the edges of the lifted ideal triangulation $\tilde{T}$ on $\widetilde{S}$ is the \textit{pleating locus} of $\Psi$. Note that the ideal vertices of the ideal triangles in $\tilde{T}$ are precisely the points in the Farey set $F_\infty$.

Given a framed representation $\hat{\rho}= (\rho, \beta)$, one can build a $\rho$-equivariant pleated plane $\Psi$ as follows: send each ideal triangle $\Delta$ of $\tilde{T}$ with ideal vertices $a,b,c$  to the totally-geodesic ideal triangle in $\mathbb{H}^3$ with ideal vertices $\beta(a), \beta(b),\beta(c)$.

If $\hat{\rho}$ is in addition, non-degenerate, then for any line of the pleating locus (i.e.\ $\Psi(\tilde{e})$ for $\tilde{e} \in \tilde{T}$) the two adjacent totally-geodesic ideal triangles $\Psi(\Delta_\pm)$  determine four distinct ideal vertices in $\cp$. Then the argument of the Fock-Goncharov coordinate $c(e)\in \mathbb{C}^\ast$ equals the angle between the two totally-geodesic planes containing the two ideal triangles $\Psi(\Delta_-)$ and $\Psi(\Delta_+)$ respectively. The real number  $\ln \lvert c(e)\rvert$, on the other hand, is the shear-parameter measuring the (signed) distance between points on the common side of $\Delta_\pm$ that are the feet of the perpendiculars from the opposite ideal vertices.

 This provides a geometric interpretation of the Fock-Goncharov coordinates.

\subsection{Alternative definitions of domination}
We note that our definition of $j$ dominating $\rho$ (see  \eqref{supl}) is equivalent to the following two definitions:
\begin{enumerate}
\item[(A)]  the inequality  in \eqref{supl} is true when the supremum is taken over \textit{all} closed curves in $\Pi$,
\item[(B)] there is a  $1$-Lipschitz map $f:\mathbb{H}^2 \to \mathbb{H}^3$ that is  $(j,\rho)$-equivariant, that is, satisfies $f\circ j(\gamma) = \rho(\gamma) \circ f$ for all $\gamma \in \Pi$.
\end{enumerate}

\noindent Note that (A) can be thought of as domination in the \textit{complete} length spectrum. \\

In fact, denote by
$$K=\sup\limits_{\gamma} \frac{l_\rho(\gamma)}{l_j(\gamma)},$$
where $\gamma$ varies over all non-peripheral essential simple closed curves (as in \eqref{supl}).
Then if $L$ is the minimal Lipschitz  constant, over all maps $f: \mathbb{H}^2\to \mathbb{H}^3$ that are $(j,\rho)$-equivariant, then $K\leq L$ (see Lemma 4.5 of \cite{GK}).

If $L\geq 1$, then by  Theorem 1.3 of \cite{GK}, 
there exists a $(j,\rho)$-equivariant $L$-Lipschitz map $f$ and  a geodesic lamination $\lambda$ on  $\mathbb{H}^2/j(\Pi)$ that is maximally
stretched, i.e., the restriction of $f $ on $\lambda$ realizes the
Lipschitz constant $L$.  Moreover, 
the maximal stretch lamination $\lambda$ is compact and contained in the convex core of $\mathbb{H}^2/j(\Pi)$.

Suppose that $L>1$, then no geodesic boundary component is in $\lambda$ since we require their lengths to remain the same.
Thus, $\lambda$ is contained in the interior of the convex core of $\mathbb{H}^2/j(\Pi)$ and it can be approximated by a sequence of non-peripheral essential simple closed geodesics.
This implies that $K >1$.
As a result, $K\leq 1$ implies that $L\leq 1$. We have shown that  \eqref{supl}, $(A)$ and $(B)$ are equivalent.
 
 \medskip

In the case of a closed surface or $\mathbb{H}^2/j(\Pi)$ is a hyperbolic surface without geodesic boundary, if $L\geq 1$, we in fact have $L=K$ by the work in \cite{GK} (see Theorem 1.3, Lemma 4.6 and the proof of Lemma 5.9).
Thus our definition of \textit{strict} domination, namely, a strict inequality in \eqref{supl}, is equivalent to the following:
\begin{enumerate}
\item[(A')]  the inequality in \eqref{supl} is strict when the supremum is taken over \textit{all} closed curves in $\Pi$,
\item[(B')]  there is an  $L <1$, and an $L$-Lipschitz map $f:\mathbb{H}^2 \to \mathbb{H}^3$ that is  $(j,\rho)$-equivariant.
\end{enumerate}

\smallskip

These alternative definitions of $j$ \textit{strictly dominating} $\rho$ are not all equivalent in the case that $\mathbb{H}^2/j(\Pi)$ is a hyperbolic surface with at least one geodesic boundary. 

As mentioned before, we always have $K \leq L$.  However, since the representation $\rho$ doesn't change the length of the boundary, 
we always have  $L\geq 1$. Thus $(B')$ does not make sense. In fact, if  $K<1$, then $L=1$ (otherwise, if $L>1$, then as we have seen above, one can show $K>1$ by approximating the maximal stretch lamination $\lambda$ by non-peripheral simple closed curves). Moreover, the result of Gueritaud-Kassel implies that the maximal stretch lamination $\lambda$ is the union of the geodesic boundary components of $j$. In addition,  $K < 1$ does not even imply (A') above, since there could be long non-simple curves most of whose length is near the boundary.

\section{Proof of Theorem \ref{thm1}}

Let $\mathcal{L} = \{l_1,l_2,\ldots l_k\}$ and suppose $\rho \in \text{Hom}(\Pi, \mathcal{L})$ is a non-Fuchsian representation.
In \S3.1 we shall assume that $\rho$ is non-degenerate; the degenerate case is handled separately in \S3.2.

\subsection{Case that $\rho$ is non-degenerate}

In this case, by Proposition 3.1 of \cite{GupGKM}, we know that there exists a framing $\beta:F_\infty \to \cp$ such that $\hat\rho = (\rho,\beta)$ is a non-degenerate \textit{framed} representation. We choose an ideal triangulation $T$ by Theorem \ref{thm-ab}, which determines Fock-Goncharov coordinates $\{c(e)\} _{e\in T} \in  (\mathbb{C}^\ast)^{\lvert T\rvert}$.

As described in \S2, these coordinates determine a $\rho$-equivariant pleated plane
\begin{equation}\label{pplane}
\Psi : \widetilde{S}  \to \mathbb{H}^3.
\end{equation}
By Proposition 3.3 of \cite{GupGKM}, the \textit{straightening} of this pleated plane results in a map $\overline{\Psi}:\widetilde{S} \to \HH^3$ whose image lies in the equatorial plane, that is a totally geodesic copy of $\HH^2$. Moreover, $\overline{\Psi}$ is the developing map of a hyperbolic surface $\widehat{S}$ homeomorphic to $S_{g,k}$, with geodesic boundaries or cusps at the punctures.   (the cusps arise exactly at the punctures corresponding to zero $\rho$-length   -- see the following lemma).  Let $j_0: \Pi \to \pslr$ denote the Fuchsian representation corresponding to $\widehat{S}$; the straightened pleated plane $\overline{\Psi}$ is then $j_0$-equivariant.

The pleating lines for $\Psi$ determine a measured lamination $\lambda$ on $\widehat{S}$, comprising a collection $\mathcal{C}$ of disjoint geodesics with weights in $(0,2\pi)$, such that each geodesic boundary component of $\widehat{S}$ has at least one leaf of $\lambda$ spiralling onto it.


Our first observation is:

\begin{lem}\label{lem0}
The $j_0$-length of the boundary curve around the $i$-th puncture $p_i$ is equal to $l_i$, for $1\leq i\leq k$. That is, $j_0 \in \text{Hom}(\Pi, \mathcal{L})$ as well.
\end{lem}
\begin{proof}
Let $s_1,s_2,\ldots ,s_k$ be the Fock-Goncharov coordinates for $\hat{\rho}$,  associated to the edges of $T$ incident on the $i$-th puncture $p_i$, and let $l_i = \sum\limits_j \ln \lvert s_j \rvert $ be their sum.  Then  by Lemma 3.2 of \cite{GupGKM}, we know that  the monodromy around the puncture $p_i$ is
\begin{itemize}
\item[(a)] loxodromic if $l_i \neq 0$
\item[(b)] parabolic or identity if $l_i = 0$ and  $\sum\limits_j Arg(s_j) \in 2\pi \mathbb{Z}  $, and
\item[(c)] elliptic if $l_i= 0$ but $\sum\limits_j Arg(s_j) \notin 2\pi \mathbb{Z} $.
\end{itemize}
Moreover, in each case, the translation length of the monodromy element (i.e.\  the $\rho$-length of the loop around $p_i$) is precisely $l_i$.  By Corollary 3.4 in \cite{GupGKM}, we then see that for the straightened surface, the monodromy around $p_i$ also has translation length $l_i$.
\end{proof}

The following geometric lemma shall be used to quantify how the translation length of a non-peripheral loop changes when we straighten:

\begin{lem}\label{trig}
For any $L>0$, $\alpha \in (0, \pi/2)$ and $\theta \in (-\pi,\pi)$, there is a constant $C>0$ such that the following holds:

Let $\mathbb{H}^2$ be isometrically embedded as the equatorial plane in $\mathbb{H}^3$, containing a geodesic segment $\ell$ and a bi-infinite geodesic line $\gamma$,  such that the two intersect at an angle at least $\alpha$, and $\ell$ has length at least $L$ on either side of $\gamma$. Let $\hat{\ell}$ be the piecewise-geodesic in $\mathbb{H}^3$ obtained when the equatorial plane is pleated along $\gamma$ by a pleating angle at least $\theta$.
Then the distance in $\mathbb{H}^3$ between the endpoints of $\hat{\ell}$ is less than $|\ell|- C$.
\end{lem}

\begin{figure}
  \centering
  \includegraphics[scale=.3]{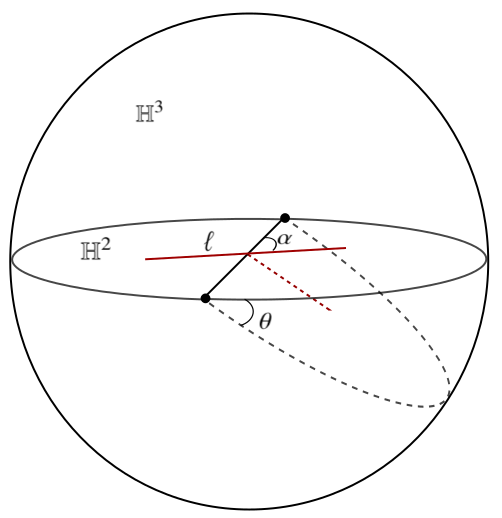}\vspace{.1in}
  \caption{The distance between the endpoints of the geodesic segment $\ell$ on the equatorial plane decreases by a definite amount, when bent along the geodesic line $\gamma$ (see Lemma \ref{trig}). }
\end{figure}

\begin{proof}
We denote by $x, y$ be the two endpoints of $\hat{\ell}$, and by $O$ be the intersection point of
${\ell}$ with $\gamma$.
Consider the geodesic triangle with vertices $x, O, y$.
The angle at $O$ is some number $\beta>0$, and it is easy to see that it depends only on the pleating angle, and the intersection angle of $\ell$ with $\gamma$. \\

\noindent\textit{Claim.} \textit{There is a constant $\delta=\delta(\alpha,\theta)>0$ such that $1+\cos\beta >\delta.$}\\
\textit{Proof of claim.}
It is enough to consider the extreme case that  the intersection angle of $\ell$ with $\gamma$
 is equal to $\alpha$, and the pleating angle is equal to $\theta$.

Assume that we are working in the unit ball model of $\mathbb{H}^3$, where $O$ is the center of the ball, $\gamma$ is a diameter with endpoints $(1,0,0)$ and $(-1,0,0)$, and $\ell$ is the entire diameter with endpoints $(\cos\alpha, -\sin\alpha, 0)$ and $(-\cos\alpha, \sin\alpha,0)$. Then the computation of $\beta$  becomes an elementary Euclidean geometry problem: after we bend the equatorial plane on one side of $\gamma$ by an angle $\theta$ (as shown in Figure 1), then $\ell$ deforms to  $\hat{\ell}$, which is the concatenation of two radial rays from $O$ with endpoints  $(\cos\alpha, -\sin\alpha, 0)$  and $(-\cos\alpha, \sin\alpha\cos\theta, - \sin \theta)$ respectively. Since the angle between them at $O$ equals $\beta$, we can compute
$$\cos \beta = (\cos\alpha, -\sin\alpha, 0) \cdot (-\cos\alpha, \sin\alpha\cos\theta, - \sin \theta) = - \cos^2(\alpha) - \sin^2(\alpha)\cos\theta$$
from which one can deduce that
$$1+\cos\beta=2\sin^2(\alpha)\sin^2(\frac{\theta}{2})$$ proving the claim. $\qed$.\\

Using the law of cosines in hyperbolic trigonometry (see, for example,  Chapter 8 of \cite{Marden}) we have:
\begin{eqnarray*}
\cosh |xy| &=& \cosh |Ox|\cosh |Oy|-\sinh |Ox|\sinh |Oy|\cos\beta.
\end{eqnarray*}
Since $|\ell| =  |Ox| + |Oy|$, we obtain
\begin{eqnarray*}\cosh|\ell|-\cosh |xy|&=&\sinh |Ox|\sinh |Oy|(1+\cos \beta)\\
&>& \sinh |Ox|\sinh |Oy|\delta
\end{eqnarray*}
Since $|Ox|,|Oy|\geq L$, and both $\sinh$ and $\cosh$ are increasing functions on the positive reals, we are done. \end{proof}

Finally, we shall need the following fact (see Lemma 2.3 of \cite{Wolpert}):

\begin{lem}[Generalized Collar Lemma]\label{hyp-lem} Given a hyperbolic surface $\widehat{X}$ of finite type, with finitely many geodesic boundaries and cusps, there exists a $D>0$ such that any non-peripheral simple closed geodesic $\gamma$ on $\widehat{X}$ remains at least distance $D$ away from the geodesic boundary components, and standard horoball-neighborhoods of the cusps.
\end{lem}

As a consequence of the distance-decreasing property in Lemma \ref{trig}, we obtain:

\begin{prop}\label{prop1}
The Fuchsian representation $j_0:\Pi \to \pslr$ dominates the representation $\rho$. Moreover, for any simple closed curve $\gamma \in \Pi$ that intersects $\lambda$ on $\hat{S}$, the $j_0$-length of $\gamma$ is strictly greater than its $\rho$-length, such that
\begin{equation*}
\sup\limits_{\gamma} \frac{l_\rho(\gamma)}{l_{j_0}(\gamma)} < 1
\end{equation*}
when $\gamma$ varies over all simple closed curves on $S_{g,k}$ that intersect $\lambda$.
 \end{prop}
\begin{proof}
Let $\gamma$ be any simple closed geodesic on $\widehat{S}$.

If the  developing image of $\tilde{\gamma}$  in the equatorial plane in $\HH^3$ does not intersect a pleating line (i.e.\ a leaf of $\lambda$), then it is not affected by the pleating, and hence the $\rho$-length will be the same as the $j_0$-length.

Else, we can decompose $\gamma$ into a  finite union of geodesic arcs $\{\gamma_j\}_{j=1}^N$,
such that each $\gamma_j$ has endpoints on $\lambda$, and has its interior  disjoint from $\lambda$.
Since the ends of leaves of $\lambda$ spiral to the $\partial \widehat{S}$ or exit out of cusps, and $\gamma$ is simple,
by Lemma \ref{hyp-lem}, $\gamma$ does not cross some collar neighborhood of $\partial \widehat{S}$ and a horodisk-neighborhoods around the cusps.

Extend $\lambda$ to a maximal ideal triangulation of $\widehat{S}$. By the observation above, each $\gamma_j$ does not lie near the cusps of the ideal triangles.  Moreover, the intersection of $\gamma$ with each leaf of $\lambda$ (a geodesic side of an ideal triangle) cannot be at an angle close to zero: if it is, $\gamma$ will remain close to the geodesic side of the ideal triangle for a large length, which would force it to lie near the cusp of that ideal triangle, contradicting Lemma \ref{hyp-lem}.

As a result, $\gamma_j$ satisfies the
hypotheses  of Lemma \ref{trig} for some $L,\alpha$ and $\theta$ (which are all independent of the choice of $\gamma$). Note that the length of  $\gamma_j$ is uniformly comparable to $L$. We denote by $|\gamma_j|=O(L)$.

The length of $\gamma$ on $\widehat{S}$ is equal to
$$l_{j_0}(\gamma)=\sum_{j}^N |\gamma_j|.$$

By Lemma \ref{trig}, there exists a $C>0$ such that
$$l_{\rho}(\gamma) < \sum_{j}^N |\gamma_j|-N C.$$ Thus
$$\frac{l_{\rho}(\gamma)}{l_{j_0}(\gamma)}\leq  \frac{\sum_{j}^N |\gamma_j|-N C}{\sum_{j}^N |\gamma_j|}<1-\frac{C}{O(L)}$$
proving the second statement.
\end{proof}

\begin{defn}[Filling arcs] A collection of pairwise-disjoint arcs on $S_{g,k}$ with endpoints at the punctures, such that each arc is homotopically non-trivial and non-peripheral is \textit{filling} if each complementary component is simply-connected. This happens, for instance, if the collection of arcs determines a maximal ideal triangulation.
\end{defn}

In what follows, we shall apply the above definition to the collection of arcs that are the leaves of the geodesic lamination $\lambda$, whose lifts to the universal cover to $\widetilde{S}$ are the pleating lines for the pleated plane $\Psi$ (see \eqref{pplane}). \\

As a consequence of the proof of Proposition \ref{prop1}, we then have:

\begin{cor}\label{cor1} If $\lambda$ is filling,  then $j_0$ strictly dominates $\rho$.
\end{cor}

Thus, in the case when $\lambda$ is filling, and $\rho$ is non-degenerate and non-Fuchsian then $j_0$ is the desired strictly-dominating Fuchsian representation $j$ in Theorem \ref{thm1}.

\subsection*{Non-filling case}
We shall now deal with the case that $\lambda$ is \textit{not} filling (we continue with our assumption that $\rho$ is non-degenerate and non-Fuchsian). Note that the assumption that $\rho$ is non-Fuchsian implies that the geodesic lamination $\lambda \neq \emptyset$. From the proof of Lemma \ref{lem0} (see also Corollary 3.4 of \cite{GupGKM}), the straightened hyperbolic surface $\widehat{S}$ has a  geodesic boundary component for each puncture whose corresponding entry in the tuple $\mathcal{L}$ is positive. Moreover, each puncture that had zero boundary length corresponds to a  (finite-volume) cusp in  $\widehat{S}$.\\

By Proposition \ref{prop1}, the holonomy $j_0$ of the hyperbolic surface $\widehat{S}$ dominates $\rho$; however the $\rho$-length and $j_0$-length  of any simple closed curve that does not intersect $\lambda$, are equal.
In what follows we shall modify $j_0$ to a \textit{strictly} dominating representation $j$. This shall use the operation of (positive)  \textit{strip-deformations}, that we shall define below. This had been used in  the proof of Lemma 3.4 of \cite{Thurston-Minimal} -- see Definition 1.3 of \cite{D-G-K}, and the proof of Lemmas 4.1 of \cite{GKW}.   \\

In what follows, let $\Sigma$ be a hyperbolic surface with at least one geodesic boundary component (and possibly some cusps), and let $\Sigma_c$ be its completion obtained by attaching a funnel to each geodesic boundary component.

\begin{defn}[Strip-deformation]\label{pstrip}  Let $\ell$ be a bi-infinite geodesic line on the complete hyperbolic surface $\Sigma_c$ such that both ends of $\ell$ exit out of funnel ends of $\Sigma_c$ (this could be the same funnel as well).   A \textit{(positive) strip-deformation with parameter $\alpha>0$} is the operation of cutting along such an infinite geodesic line, say $\ell$ and inserting a ``strip", a  hyperbolic region bounded by two disjoint geodesics at a minimum distance $\alpha>0$. See Figure 2.
\end{defn}

\begin{figure}
  \centering
  \includegraphics[scale=.32]{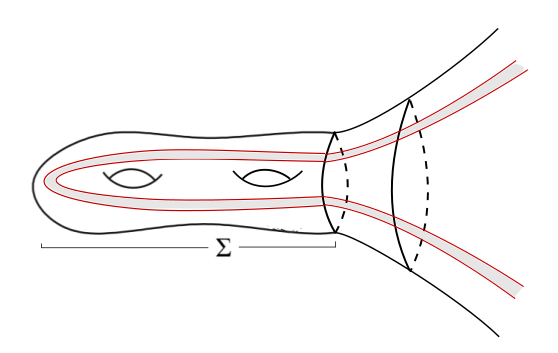}
  \caption{A strip-deformation inserts a hyperbolic strip along a geodesic line embedded in the completion of $\Sigma$. }
\end{figure}

Using positive strip-deformations, one can obtain the following result (see Lemma 4.1 of \cite{GKW}, and \cite{Parlier} for a weaker result):

\begin{prop}\label{inc} Given a hyperbolic surface  $\Sigma$ with non-empty geodesic boundary, there exists a hyperbolic surface $\Sigma^\prime$ homeomorphic to $\Sigma$, such that
\begin{equation}\label{supl2}
\sup\limits_{\gamma} \frac{l_\Sigma(\gamma)}{l_{\Sigma^\prime}(\gamma)} < 1
\end{equation}
where $l_X(\gamma)$ denotes the hyperbolic length of the (geodesic representative of) the curve $\gamma$ on the hyperbolic surface $X$, and $\gamma$ varies over all simple closed curves, including the boundary components.
\end{prop}
\begin{proof}[Sketch of the proof] As before, extend $\Sigma$ to a complete hyperbolic surface $\Sigma_c$  by adding funnels to each geodesic boundary component. Choose a collection $\mathcal{L}$ of pairwise-disjoint embedded bi-infinite geodesic lines on $\Sigma_c$ that are filling, and perform a positive strip-deformation (with some positive parameter) on each. The surface $\Sigma^\prime$ is obtained by excising the funnels of the resulting complete hyperbolic surface. The length of a simple closed curve  increases by at least  the width of the added strip each time  it crosses an arc in $\mathcal{L}$. Since any simple closed curve, including the boundary components, intersects $\mathcal{L}$, its length in $\Sigma^\prime$ increases by a definite factor, exactly as in the proof of Proposition \ref{prop1}.
\end{proof}

\textit{Remarks}. 1. The same proof, with \textit{negative} strip-deformations, allows one to shorten lengths of all simple closed curves on a hyperbolic surface with boundary. For details, see Lemma 4.4 of \cite {GKW}, or \cite{P-T} (where the operation is called ``peeling a strip").

2. In the construction outlined in the proof of Proposition \ref{inc}, the length of some simple geodesic arc between boundary component(s) of $\Sigma$ would necessarily decrease. This is because  by doubling $\Sigma$ across its geodesic boundaries we would obtain a closed surface, and we know that we cannot lengthen all simple closed geodesics on a closed hyperbolic surface (\textit{c.f.} Proposition 2.1 of \cite{Thurston-Minimal}). In fact, in the proof of Proposition \ref{main-prop} below, we shall see that for sufficiently small strip-deformations on $\Sigma$, the lengths of simple geodesic arcs between boundary component(s) of $\Sigma$ decrease by a \textit{uniformly} small amount i.e.\ by a multiplicative factor that is bounded below away from $0$.  \\

We shall apply strip-deformations and Proposition 3.3 to various subsurfaces of $\widehat{S}$ in the proof of the following:

\begin{prop}\label{main-prop} There is a hyperbolic surface $\widehat{S}_t$ with the underlying topological surface $S_{g,k}$, such that
\begin{itemize}
\item[(A)] like $\widehat{S}$,  the $i$-th puncture is a cusp if the corresponding entry of $\mathcal{L} = (l_1,l_2,\ldots, l_k)$ is zero, and a geodesic boundary component of length $l_i$ otherwise, and
\item[(B)] the lengths of simple closed geodesics on  $\widehat{S}_t$ are greater than the corresponding geodesics on $\widehat{S}$ in a way that
\begin{equation}\label{supl3}
\sup\limits_{\gamma} \frac{l_{\rho}(\gamma)}{l_{\widehat{S}_t}(\gamma)} < 1
\end{equation}
where $\gamma$ varies over all non-peripheral simple closed curves on $S_{g,k}$.
\end{itemize}
\end{prop}

\begin{proof}
In what follows, the surface  $\widehat{S}_t$  will be constructed by a suitable deformation of $\widehat{S}$, involving the complementary components of $\lambda$ on $\widehat{S}$ that are not simply-connected. (Recall that such components exists since $\lambda$ is not filling.) \\

Namely, let $\Sigma^\circ$ be a connected component of $\widehat{S} \setminus \lambda$ that is not simply-connected. Then the metric completion of $\Sigma^\circ$ is a ``crowned" hyperbolic surface. Recall that a \textit{crowned hyperbolic surface} is a hyperbolic structure on a punctured surface $S$, such that each puncture corresponds to a ``crown end", bounded by chain of bi-infinite geodesic lines arranged in a cyclic order, such that the positive half-ray of each line and the negative half-ray of the next are asymptotic.

The assumption that  $\Sigma^\circ$  is not simply-connected implies that it is not an ideal polygon. In this case, there is a geodesic representative of the loop around each crown end, and these collection of loops typically bound an embedded hyperbolic surface with geodesic boundary (the \textit{convex core} of $\Sigma^\circ$), that we denote by $\Sigma$.  (See Figure 3.) The exceptional case is when $\Sigma^\circ$ is topologically an annulus, with exactly two crown ends, in which case the convex core  $\Sigma$ is a single simple closed geodesic homotopic to the loop around either end. \\

In this way, we obtain a collection of hyperbolic surfaces with geodesic boundary that we denote by $\Sigma_1, \Sigma_2,\ldots , \Sigma_l$  and possibly some simple closed geodesics (as in the exceptional case), that we denote by  $\sigma_1,\sigma_2,\ldots \sigma_m$. Here each $\Sigma_i$ is connected and embedded in $\widehat{S}$, and the $\Sigma_i$s and $\sigma_j$s are all pairwise disjoint.
Consider now the (possibly disconnected) hyperbolic surface
\begin{equation}\label{surfR}
R= \widehat{S} \setminus \left(\Sigma_1 \cup \Sigma_2 \cup \cdots \cup \Sigma_l \cup \sigma_1 \cup \sigma_2 \cup \cdots \cup \sigma_l\right).
\end{equation}
 Note that this surface contains all the cusps and/or geodesic boundary components of $\widehat{S}$.  We call the geodesic boundary components of $R$ that arise from the boundary components of the $\Sigma_i$s and the $\sigma_j$s to be the \textit{positive} boundary components.\\

 \begin{figure}
  \centering
  \includegraphics[scale=.32]{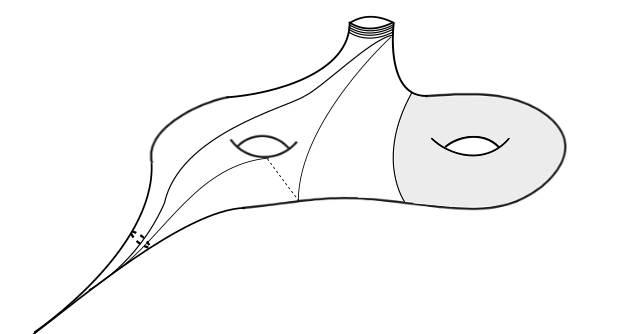}
  \caption{The hyperbolic surface $\widehat{S}$ with a subsurface $\Sigma$ (shown shaded) in the complement of $\lambda$. }
\end{figure}

We shall now modify $R$ to a new (topologically identical) hyperbolic surface $R_t$ by the following operation: choose a collection of pairwise-disjoint non-peripheral geodesic arcs between the \textit{positive} boundary components of $R$, such that each positive boundary component has a least one incident arc, and each arc is perpendicular to the positive boundary components at its endpoints. (The shortest arc in a non-trivial homotopy class between boundary components would have the last property.)  Let $R_c$ be the completion of $R$ obtained by attaching hyperbolic funnels to each of the positive boundary components.  These arcs can now be extended to bi-infinite geodesic lines in $R_c$ that are still pairwise-disjoint (if they intersect, we would obtain a hyperbolic triangle with two right angles, bounded by segments of the two intersecting lines and an arc in a positive boundary component, which is impossible). Hence, we can perform a positive strip deformation with parameter $t>0$ on each of these geodesic lines, to obtain a new hyperbolic surface. $R_t$ is the surface obtained when we remove the funnels from this new surface. By construction, the lengths of all the positive boundary components of $R$ have increased;  we continue to call these the positive boundary components of $R_t$. Note that the lengths of the other boundary components of $R$ are not affected in this deformation. The parameter $t$ in the above discussion will be chosen in the forthcoming discussion (see the Claim below).  \\

First, by applying Proposition \ref{inc} to each $\Sigma_i$, we obtain a topologically identical hyperbolic surface $\Sigma_i^\prime$ such that the lengths of all simple closed geodesics in $\Sigma_i^\prime$, including the geodesic boundary components, are greater (by a multiplicative factor greater than $1$) than the corresponding lengths on $\Sigma_i$. Moreover, we can choose the positive parameters on the set of filling arcs in the proof of Proposition \ref{inc}, to  make sure that the lengths of the boundary components of each $\Sigma_i^\prime$  match with the lengths of the corresponding positive boundary components of $R_t$. Then we obtain a new hyperbolic surface $\widehat{S}_t$ by gluing (a) each $\Sigma_i^\prime$ to $R_t$, and (b) pairs of boundary components of $R_t$ corresponding to the $\sigma_j$s, exactly as dictated by the identification of the boundaries of $\Sigma_i$s and the $\sigma_j$s  on $\widehat{S}$ (see \eqref{surfR}). (In particular, we keep the twists parameters of each gluing exactly the same as we see on $\widehat{S}$.)  \\

By our construction, the surface $\widehat{S}_t$ satisfies the requirement (A) in the statement of the Proposition, since the deformations described above do not affect the cusps and geodesic boundary components of $\widehat{S}$. We shall now show that we can choose the parameter $t$ so that the requirement (B) of the Proposition is also satisfied.

Let $c$ be a simple closed curve on $\widehat{S}$, that is a boundary component of some $\Sigma_i$, or one of the $\sigma_j$s. Note that $c$ has an annular neighborhood in $\widehat{S}$,  that is bounded by $c$ on one side, and a collection of leaves of $\lambda$ on the other. (The latter is the crown end of the complementary component of $\lambda$  that $c$ is contained in.) This implies that any simple closed geodesic $\gamma$ on $\widehat{S}$ that intersects $c$ essentially, must necessarily intersect $\lambda$.  Indeed, $\gamma$ must intersect $\lambda$ at least as many times as it intersects $\partial \Sigma_i$.  Thus, by Proposition \ref{prop1}, we know that
\begin{equation}\label{supl4}
\sup\limits_{\gamma} \frac{l_{\rho}(\gamma)}{l_{\widehat{S}}(\gamma)} = \beta < 1
\end{equation}
when $\gamma$ varies over all  simple closed geodesics that are not completely contained in one of the embedded $\Sigma_i$s, and is not one of the $\sigma_j$s.

The key observation now is that the positive strip-deformations resulting in $\widehat{S}_t$ changes the fundamental domain for $\widehat{S}$  to that of $\widehat{S}_t$ continuously, as a function of $t$. As a consequence, we derive the following claim (see  Remark (2) after Proposition \ref{inc}):\\

\noindent\textit{Claim.} \textit{For any sufficiently small $t>0$, we have}
\begin{equation}\label{supl5}
\sup\limits_{\gamma} \frac{l_{\widehat{S}}(\gamma)}{l_{\widehat{S}_t}(\gamma)} <  1/\beta
\end{equation}
\textit{when $\gamma$ varies over all  simple closed curves that are not completely contained in one of the embedded $\Sigma_i$s,  and are not the $\sigma_j$s. (Here $\beta$ is the constant less than $1$ obtained in \eqref{supl4}.)}\\
\textit{Proof of claim.}
Any simple closed geodesic $\gamma$ on $\widehat{S}$ as above can be decomposed into a finite union of geodesic arcs $\{\gamma_j\}_{j=1}^N$ such that each $\gamma_j$ is a either a geodesic arc between a boundary component of a $\Sigma_i$ (for some $1\leq i\leq l$) to itself, or a geodesic arc with endpoints on two (or possibly the same) positive boundary components of $R$. In either case, the length of each $\gamma_j$, denoted $|\gamma_j|$, is bounded below by some $L>0$, that depends on the hyperbolic surface $\widehat{S}$ and its subsurfaces $\Sigma_i$s and $R$. Since the fundamental domains for $\widehat{S}_t$  and its corresponding subsurfaces in $\mathbb{H}^2$ change continuously as we increase $t$ from $0$, the length of each $\gamma_j$ changes continuously. Thus, for any choice of $c>0$, we can choose $t$ small enough such that this difference of lengths is bounded by $c$ (for each $j$).
Then we have $$l_{\widehat{S}_t}(\gamma) > \sum_{j=1}^N |\gamma_j| - Nc.$$
and consequently
$$\frac{l_{\widehat{S}}(\gamma)}{l_{\widehat{S}_t}(\gamma)}\leq  \frac{\sum_{j=1}^N |\gamma_j|}{\sum_{j=1}^N |\gamma_j|-Nc}<1 +\frac{Nc}{\sum_{j=1}^N |\gamma_j|} < 1 + \frac{c}{L}.$$
To obtain \eqref{supl5}, we choose $c$ such that the right hand side above is equal to $1/\beta$. $\qed$.\\

Combining \eqref{supl4} and \eqref{supl5}, we see that the inequality \eqref{supl3} in requirement (B) of the Proposition holds when the supremum is taken over all simple closed curves $\gamma$ that are not completely contained in one of the embedded $\Sigma_i$s, and is not one of the $\sigma_j$s.

However, if $\gamma$ is a simple closed curve contained entirely in one of the $\Sigma_i$s, or is one of the $\sigma_j$s,  then $l_\rho(\gamma)$ = $l_{\widehat{S}}(\gamma)$ since $\gamma$ is disjoint from $\lambda$.  Also, the length of $\gamma$ on $\widehat{S}_t$ is equal to the length of $\gamma$ on the embedded subsurface $\Sigma_i^\prime$. Hence by Proposition \ref{inc}, we obtain
\begin{equation*}
\sup\limits_{\gamma} \frac{l_{\rho}(\gamma)}{l_{\widehat{S}_t}(\gamma)} < 1
\end{equation*}
when the supremum is taken over all such simple closed curves.

Thus, \eqref{supl3} holds when the supremum is taken over all simple closed curves on $S_{g,k}$, and requirement (B) is satisfied.  \end{proof}

\medskip
As a consequence, the holonomy $j:\Pi \to \pslr$ of the hyperbolic surface  $\widehat{S}_t$ (obtained in Proposition \ref{main-prop}) strictly dominates the non-degenerate and non-Fuchsian representation $\rho$ we started with, in the  beginning of the section.

\subsection{Case that $\rho$ is degenerate}

We now handle the remaining case when $\rho$ is a degenerate representation. Recall from Definition \ref{degen} that there are two possibilities (a) and (b), where the image has exactly one and two global fixed points on $\cp$, respectively.

In the case that (a) in Definition \ref{degen} holds,  the monodromy around each puncture of $S_{g,k}$ is parabolic, and the representation  $\rho$ lies in the relative character variety $\text{Hom}(\Pi,  \mathcal{L})$ where $\mathcal{L} = (0,0,\ldots ,0)$. Then, let $j:\Pi \to \pslr$ be any Fuchsian representation such that each of the $k$ punctures is a cusp. Then $j \in \text{Hom}(\Pi,  \mathcal{L})$, and strictly dominates $\rho$; indeed, the left hand side of \eqref{supl} then equals zero. 

Finally, in the case that $\rho$ is degenerate and co-axial (i.e.\  (b) in Definition \ref{degen} holds), then  the image of $\rho$ 
preserves the geodesic line $\ell$ in $\mathbb{H}^3$,  which by a  conjugation can be assumed to be to a geodesic line in the equatorial plane $\mathbb{H}^2$. If we identify $\ell$ with $\mathbb{R}$, via an isometry $\Psi$,  then each element in the image of $\rho$ acts by a half-translation along $\mathbb{R}$ (i.e\ has the form $x\mapsto \pm x+c$) . Thus, for each $\gamma \in \Pi$,  there is a real number $m(\rho(\gamma)) = \Psi(\rho(\gamma)\cdot x) - \Psi(x)$, that is well-defined, i.e.\ independent of $x$. Moreover, these satisfy:  (i) $l_\rho(\gamma) = \lvert m(\rho(\gamma)) \rvert$ and (ii) $m \circ \rho :\Pi \to \mathbb{R}$  is a homomorphism. 
Then, as in \cite{DerThol}, this homomorphism to $\mathbb{R}$  can be considered as a defining a virtually-abelian representation  $\rho^\prime:\Pi \to \pslr$ that preserves a geodesic line $\ell$ in $\mathbb{H}^2$, and acts by translations along it. Note that the translation distance of $\rho^\prime(\gamma)$ along $\mathbb{R}$  is exactly $l_\rho(\gamma)$, for each $\gamma \in \Pi$.

Thus, it suffices to find a Fuchsian representation $j$ that is strictly dominates $\rho^\prime$ in the sense of \eqref{supl}.  To do this, we can apply the techniques of either  \cite{GKW} or \cite{Sagman}. We thank Nathaniel Sagman for the following sketch of the latter approach:

Let $\mathcal{L} = (l_1,l_2,\ldots,l_k)$ be the $\rho^\prime$-lengths of the loops around the punctures of $S$. 
Recall that we had chosen a hyperbolic metric of finite volume on $S= S_{g,k}$ in \S2. By Theorem 1.1 of \cite{Sagman}, there is a $\rho^\prime$-equivariant map $f:\widetilde{S} \to \mathbb{H}^2$ with image $\ell$ such that the Hopf differential $\text{Hopf}(f)$ on $S$ has a pole at the $i$-th puncture of order at most one  if $l_i=0$, and of order two if $l_i>0$, with a real residue determined by $l_i$. Moreover, by Theorem 1.4 of \cite{Sagman} there is a hyperbolic surface $S^\prime$ of finite volume that is homeomorphic to $S$, with boundary-lengths given by $\mathcal{L}$, and a harmonic map $h:S\to S^\prime$ such that $\text{Hopf}(h) = \text{Hopf}(f)$.

 By Proposition 3.13 of \cite{Sagman},  the energy densities satisfy $e(f) < e(h)$ pointwise, everywhere on $S$. Moreover, as one approaches a puncture of $S$, the ratio $e(f)/e(h) \to 1$, if the monodromy around it is parabolic or hyperbolic (by Theorem 1.1 of \cite{Sagman} and Proposition 3.13 of \cite{WolfInf}) and $e(f)/e(h) \to 0$, if the monodromy around the puncture is elliptic (by Proposition 6.1 of \cite{Sagman}). Then, if $j$ is the Fuchsian holonomy of $S^\prime$, the $(j,\rho)$-equivariant map $f\circ \tilde{h^{-1}}:\widetilde{h}(\mathbb{H}^2) \to \mathbb{H}^2$ is strictly $1$-Lipschitz on any compact subset of  $S^\prime$. Since any simple closed geodesic on $S^\prime$ lies in a compact subset of $S^\prime$ by Lemma \ref{hyp-lem}, it follows that $j$ strictly dominates $\rho$ in the simple length spectrum, as in \eqref{supl}. 

This concludes the proof of Theorem \ref{thm1}.

\bibliographystyle{amsalpha}
\bibliography{mrefs}

\end{document}